\documentclass[12pt]{amsart}

\usepackage[latin1]{inputenc}
\usepackage[english]{babel}
\usepackage{amsthm}
\usepackage{amsmath}
\usepackage{amsxtra}
\usepackage{mathrsfs}
\usepackage{amssymb}
\usepackage[dvips]{graphicx}
\usepackage{graphicx}
\usepackage{mathabx}

\input xy
\xyoption{all}

\newtheorem{teo}{Theorem}[section]

\newtheorem{defi}[teo]{Definition}
\newtheorem{coro}[teo]{Corollary}
\newtheorem{fatto}[teo]{Fact}

\newtheorem{prop}[teo]{Proposition}

\theoremstyle{remark}
\newtheorem{remark}[teo]{Remark}
\newtheorem{notation}[teo]{Notation}

\begin{document}

\title[Infinite forcing and the generic multiverse]{Infinite forcing and the generic multiverse}
\author{Giorgio Venturi}
\email{gio.venturi@gmail.com}
\address{Philosophy department Unicamp
Rua Cora Coralina, 100
Barão Geraldo, SP.}


\maketitle

\begin{abstract}
In this article we present a technique for selecting models of set theory that are complete in a model-theoretic sense. Specifically, we will apply Robinson infinite forcing to the collections of models of ZFC obtained by Cohen forcing. This technique will be used to suggest a unified perspective on generic absoluteness principles. 
\end{abstract}

\vspace{0.5 cm}

\vspace{0.5 cm}

\section*{Introduction}

Soon after the invention of forcing, in 1963, the abundance of independence results started to undermine
the hope for a unified picture of the universe of all sets. An exemplar reaction is the formalist perspective taken 
by Cohen himself in the more philosophical interpretation of his celebrated result on the independence of the Continuum Hypothesis. 

\begin{quote}
So, let me say that I will ascribe to Skolem a view, not explicitly stated by him, that there is a reality to mathematics, but axioms cannot describe it. Indeed one goes further and says that there is no reason to think that any axiom system can adequately describe it.\footnote{\cite{Cohen2005}, p. 2417.}
\end{quote}

But forcing, 
by itself, does not disprove a realist perspective on truth in set theory. On the contrary, the deep understanding and refinement
of the method helped the search for forcing invariant theories \cite{wooBMMc}, \cite{Viale2}, 
with the goal to isolate the right axiomatic extensions of ZFC.

The question whether forcing  suggests to abandon or allows to strengthen our search for a unified picture  
does not have a clear answer. 
Moreover, even on the side of a realist interpretation of mathematical reality, the debate on competing extensions of ZFC shows that 
choosing the right criteria to overcome independence is a complex task, where not only mathematical considerations
need to be taken into account \cite{FriedmanArrigoni}, 
\cite{Magidor}, \cite{Vaananenmultiverse}, \cite{WoodinInfinite}.

We propose here to accept a generic perspective on mathematical 
truth, without nevertheless embracing position
\emph{à la} Balaguer-Hamkins \cite{Hamkins}, \cite{Balaguer}, akin to a kind of formalism. 
Indeed, the starting point of this work consists in accepting independence as an essential limitation of our axiomatic method.
This means that, in order to complete ZFC, we should not only try to extend the axioms of set theroy, but 
we also need semantic considerations for the selection of the right models. But on which ground should we choose?

Tools for the semantic completion of mathematical theories have been largely studied in model theory. 
We propose here to apply a technique that has proven to be successful in other fields: \emph{Robinson infinite forcing}. 
This method originates from the study of model complete theories and generalizes the notion of algebraically
closed field to the algebraic study of other classes of structures. Differently from Cohen's forcing, Robinson's  
infinite forcing does not build new models but it selects the generic ones out of a given class of structures. 

In this paper we investigate the possibility to apply Robinson infinite forcing to the so-called  \emph{generic multiverse}:
the collection of ZFC models obtained by Cohen forcing. This will allow us to use
algebraic considerations for the selection of the relevant generic models of set theory. Interestingly, 
these generic models will  show a high degree of genericity, offering a unified picture of generic absoluteness principles like 
Bounded Forcing Axioms, Resurrection Axioms, and Maximality Principles. 
If any, the importance of this perspective 
is given by the unified picture it offers.

 The paper is structured as follows: in \S 1.1 we briefly present the general theory of Robinson infinite forcing, that we 
 apply, in \S 1.2, to the generic multiverse of ZFC. Finally, in \S 1.3, we offer two concrete examples of generic models
 of set theory. We conclude, in \S 2, with plans for future work.

\section{Robinson infinite forcing of Cohen generic multiverse}

The basic structure we study in this paper is the collection of generic extensions of a countable transitive model $M$; c.t.m. henceforth.

\begin{defi}
Given a c.t.m. $M$ of ZFC and a class of forcings $\Gamma$, defined by a first order formula $\varphi_{\Gamma}$ in the language of set theory,
we let $\mathbb{M}_{M}^\Gamma$ be the collection of all generic extensions of $M$ by elements of $\Gamma^{M}$, i.e., the
extension of the relativization of $\varphi_{\Gamma}$ to $M$. Moreover for 
$N, N' \in \mathbb{M}_{M}^\Gamma$ we set $N' \leq_{\mathbb{M}_{M}^\Gamma} N$, whenever $N'$ is a generic extension of $N$
by means of a notion of forcing $\mathbb{P} \in \Gamma^N$.
\end{defi}

We call $\mathbb{M}_{M}^\Gamma$  the \emph{$\Gamma$ set-generic multiverse of $M$}. Notice that if a class of 
posets $\Gamma$ contains the trivial forcing and is closed under two steps iterations, we can consider a multiverse 
$\mathbb{M}_{M}^\Gamma$ as a partially ordered set. 

\begin{notation}
We indicate the Cohen forcing for adding one Cohen real as 
$Add(\omega, 1)$, while with $Add(\omega, \omega)$ we refer to its isomorphic copy which adds
$\omega$-many Cohen reals. We regard $Add(\omega, 1)$ as the partial order of finite binary sequences 
in $2^{<\omega}$ and each Cohen real $c$ as an element of the Cantor space $2^\omega$. 
\end{notation}

We can now define the two main mutliverses we will study in \S 1.3.

\begin{defi}
By $\mathbb{M}_{M}^C$ we indicate the $\Gamma$ set-generic multiverse of $M$, where $\Gamma$ consists of the Cohen forcing $Add(\omega, 1)$ and the trivial forcing.  Moreover, by $\mathbb{M}_{M}^{\kappa}$ we indicate  the multiverse consisting  of all generic
extensions of $M$ by forcing of cardinality $\leq \kappa^M$. 
\end{defi}

We call $\mathbb{M}_{M}^C$ the \emph{Cohen generic multiverse of $M$}, while $\mathbb{M}_{M}^{\kappa}$ the \emph{$\kappa$ generic multiverse of $M$}.

 \subsection{Robinson's forcing}

 The central motivation underlying Robinson's forcing\footnote{See \cite{HW} and \cite{Hodges} for a general presentation.} is not set theory, but algebra. In particular, Robinson was interested in finding analogs of algebraically closed fields within other classes of structures. 
To this aim Robinson introduced minimal conditions which must be satisfied in order for a structure to have ``algebraically closed'' extensions. 

Robinson defined two distinct model-theoretic notions of forcing, namely, infinite and finite forcing. The latter resembles more closely Cohen's forcing and
it will not be discussed in this work. On the other hand, infinite forcing characterizes a concept that is closely related to the notion of
existential completeness---i.e., when $\Sigma_1$-formulas are absolute between a structure and its extensions---and
 of model-companionship, which is defined as follows.

\begin{defi}
Given two first order theories $T$ and $T^*$, expressed in the same language, we say that $T^*$ is the \emph{model companion}
of $T$ iff  (a) $T$ and $T^*$ are mutually consistent (i.e., every model of $T$ is embeddable in a model of $T^*$ and vice versa)
and (b) $T^*$ is model complete (i.e., every embedding between two models of $T^*$ is elementary). 
\end{defi}

In order to apply Robinson infinite forcing to the Cohen generic multiverse, we recall few definitions and facts. 

\begin{defi}\label{InfRobFor}
Infinite forcing is a relation between a structure $M$ and a formula $\varphi$, with (a possibly empty set of) parameters in $M$. In symbols $M \VDash_i \varphi$. In words $M$ $i$-forces $\varphi$. It can be recursively characterized as follows. 

\begin{enumerate}
\item if $\varphi$ is atomic, then  $M \VDash_i \varphi$ if and only if  $M \models \varphi$,
\item $M \VDash_i  \varphi \land \psi$ if and only if $M \VDash_i \varphi$ and $M \VDash_i \psi$,
\item $M \VDash_i  \varphi \lor \psi$ if and only if $M \VDash_i \varphi$ or $M \models_i \psi$,
\item $M \VDash_i \exists x\varphi(x)$  if and only if $M \VDash_i \varphi(a)$, for some element $a \in M$,
\item $M \VDash_i \neg \varphi$  if and only if, for any extension $N$ of $M$, $N$ does not $i$-force $\varphi$. 
\end{enumerate}
\end{defi}

As usual we consider other connectives like $\forall$ and $\to$ as abbreviations.

\begin{remark}
In Definition \ref{InfRobFor} we are implicitly assuming the existence of an appropriate interpretation, in $M$, of the language in which 
a formula $\varphi$ is expressed. For the sake of simplicity we will continue to avoid this level of details, since it is not relevant for 
the arguments of this paper. 
\end{remark}

\begin{defi}
We say that  a structure $M$ is \emph{infinitely generic} whenever, for any formula $\varphi$, it is the case that $M \VDash_i \varphi$ or $M \VDash_i \neg\varphi$. 
\end{defi}

Infinitely generic structures are those in which forcability and satisfiability coincide.  

\begin{teo}(\cite{HW}, Theorem 3.3)\label{Infgen}
If a structure $M$ is infinitely generic, then 
 $M \models \varphi$ iff $M \VDash_i \varphi$, for every formula $\varphi$. 
\end{teo}

Under specific conditions, infinitely generic structures always exist. 

\begin{teo}\label{inductive}(\cite{HW}, Proposition 3.2)
If a class of structures $\Sigma$ is inductive (i.e. the union of structures in $\Sigma$ results in a structure in $\Sigma$), then  any structure in $\Sigma$ can be extended to an infinitely generic structure  in $\Sigma$. 
\end{teo}

\subsection{Modified Robinson's forcing}

We can now state the relevant property that 
we need to impose on the multiverse in order to apply Robinson infinite forcing. We need to select a $\Gamma$ such that 
$\mathbb{M}_{M}^\Gamma$ is $\sigma$-closed.

Even though the elements of  $\mathbb{M}_{M}^\Gamma$ are not closed under arbitrary unions, $\sigma$-closeness will provide
enough directness for Robinson infinite forcing construction to go through. Towards this end we can define the relevant modification 
of Robinson infinite forcing; the only difference is that we consider generic extensions instead of $\subseteq$-extensions.

\begin{defi}\label{Robinsongeneric}
The relation \emph{$N$ infinitely forces $\varphi$ in the $M$-multiverse}, denoted $N \VDash_M \varphi$, between a structure $N \in \mathbb{M}_{M}^\Gamma$  
and a sentence $\varphi$ of the language of set theory, with (a possibly empty set of) parameters in $N$, is defined as follows:
\begin{enumerate}

\item if $\varphi$ is atomic, then  $N \VDash_M \varphi$ if and only if  $N \models \varphi$,
\item $N \VDash_M  \varphi \land \psi$ if and only if $N \VDash_M \varphi$ and $N \VDash_M \psi$,
\item $N \VDash_M  \varphi \lor \psi$ if and only if $N \VDash_M \varphi$ or $N \models_M \psi$,
\item $N \VDash_M \exists x\varphi(x)$  if and only if $N \VDash_M \varphi(a)$, for some element $a \in N$,
\item $N \VDash_ M \neg \varphi$  if and only if  there is no $Q$ such that  $Q \leq_{\mathbb{M}_{M}^\Gamma} N$ and $Q \VDash_M \varphi$.

 \end{enumerate}
\end{defi}

As in Definition \ref{InfRobFor}, other connectives like $\forall$ and $\to$ are treated as abbreviations.

\begin{remark}
For notational convenience we will drop the subscript $M$ from the $\VDash_M$-relation, since we are not interested here in specific multiverses.
Therefore we will refer to Definition \ref{Robinsongeneric} as the \emph{infinite multiverse forcing} and not as the infinite $M$-multiverse 
forcing.  However it should be kept in mind that the definition of the infinite multiverse forcing may not be absolute\footnote{On the lack
of absoluteness of Robinson's infinite forcing see \cite{Manevitz}.}. 
\end{remark}

The following facts are easy consequences of the definition of the infinitely generic forcing and show similarities of this notion both with Cohen forcing 
and with the satisfaction relation. 
 
 \begin{fatto}
If $\varphi$ is a sentence of the language of set theory, with parameters in a model $N$, then $N$ cannot infinitely generic force both $\varphi$ and $\neg \varphi$. \qed
 \end{fatto}

 \begin{fatto}
If $\varphi$ is a sentence of the language of set theory, with parameters in a model $N$, such that   $N \VDash \varphi$, and $Q$ is such that  $Q \leq_{\mathbb{M}_{M}^\Gamma} N$,
then $Q \VDash \varphi$. \qed 
 \end{fatto}

Our main interest in Robinson forcing is given by the possibility to define structures that are generic in a model theoretic sense. The following definition is
meant to capture a notion that is a mixture of Cohen's forcing and Robinson's infinite forcing. 

\begin{defi}
An element $N$ of $\mathbb{M}_{M}^\Gamma$  will be called \emph{infinitely generic in the multiverse} if for each sentence $\varphi$ of the language of set theory, with parameters in  $N$,
either $N \VDash \varphi$ or $N \VDash \neg \varphi$. 
\end{defi}

Structures that are infinitely generic in the multiverse not only exist, but are dense in the partial order induced by the relation of generic extendability. 

\begin{teo}\label{generic}
Every model in $\mathbb{M}_{M}^\Gamma$ is a substructure of a model that is infinitely generic in the multiverse.  
\end{teo}
\begin{proof}
Let $N$ be an element of $\mathbb{M}_{M}^\Gamma$ and let $\{\varphi_n : n \in \omega\}$ be an enumeration of all the sentences in the language of set theory, 
with parameters in $N$.  Set $N_0 = N$ and define inductively the following chain in $\mathbb{M}_{M}^\Gamma$. Let $N_m$ be defined, if $N_m \VDash \varphi_m$
or $N_m \VDash \neg \varphi_m$, then let $N_{m+1} = N_m$. Otherwise, since $N_m \not\VDash \neg \varphi_m$, there is an an $N'$ such that $N' \leq_{\mathbb{M}_{M}^\Gamma} N_m$
and $N' \VDash  \varphi_m$. Then, let $N_{m+1} = N'$. 

Now, let $N^1$ be a grater lower bound of  $\{N_m : m \in \omega\}$. By $\sigma$-closeness, $N^1$ is still a member 
of $\mathbb{M}_{M}^\Gamma$. Moreover, if $\varphi$ is a formula with parameters in $N$, either $N^1 \VDash \varphi$ or $N^1 \VDash \neg \varphi$. In order to make an extension of $N$ fully generic we continue this construction and we generate a sequence of models $\{N^m : m \in \omega\}$ such that 

\begin{enumerate}

\item for all $m \in \omega$, we have $N^{m+1} \leq_{\mathbb{M}_{M}^\Gamma} N^m$, and
\item if $\varphi$ is a  a sentence, with parameters in a model $N_m$, then either $N^{m+1} \VDash \varphi$ or $N^{m+1} \VDash \neg \varphi$. 

 \end{enumerate}
 
 Then let $N^{\omega}$  be a grater lower bound of $\{N_m : m \in \omega\}$. 
 By construction $N^{\omega} \leq_{\mathbb{M}_{M}^\Gamma} N$ and it is infinitely generic in the multiverse. 
\end{proof}

It is possible to prove an analog of Theorem \ref{Infgen} in this new setting.

\begin{teo}(\cite{HW}, Theorem 3.3, essentially)
A model $N \in \mathbb{M}_{M}^\Gamma$ is infinitely generic in the multiverse if and only if, for each sentence $\varphi$ with parameters in  $N$, the following 
holds: $ N \VDash \varphi \mbox{ if and only if } N \models \varphi.$ \qed
\end{teo}

While Cohen forcing produces models that can differ substantially between each others, Robinson's forcing produces models that display a very stable theory.

\begin{prop}\label{geneq}
If $N$ is a structure that is infinitely generic in the multiverse and $N'$ is a second structure that is infinitely generic in the multiverse and such that 
$N' \leq_{\mathbb{M}_{M}^\Gamma} N$, then $N \prec N'$: i.e., $N$ is an elementary substructure of $N'$. 
\end{prop}
\begin{proof}
Let $\varphi$ be a sentence with parameters in $N$. If $N \models \varphi$, then $N \VDash \varphi$. So $N' \VDash \varphi$ and therefore $N' \models \varphi$.
On the other hand if $N' \models \varphi$, then $N' \VDash \varphi$. Now, since $N \VDash \neg\varphi$ would contradict that $N' \VDash \varphi$, then, due to
the genericity of $N$, we have that $N \VDash \varphi$. This is enough since we get that $N \models \varphi$. 
\end{proof}

Proposition \ref{geneq} is the key ingredient to show the high level of absoluteness displayed by the structures that are infinitely generic in the multiverse. 
In order to make the above statement more precise we need a couple of definitions that are adaptations of analogous notions from model theory.

\begin{defi}
Given $N$ and $N'$ such that $N' \leq_{\mathbb{M}_{M}^\Gamma} N$, $N$ is said to be  \emph{existentially complete in $N'$}, 
if every $\Sigma_1$-sentence with parameters in $N$ and true in $N'$ is also true in $N$. Moreover, we say that $N$ 
is \emph{existentially complete in $\mathbb{M}_{M}^\Gamma$} if $N$ is existentially complete in $N'$ for every $N'$ such that
$N' \leq_{\mathbb{M}_{M}^\Gamma} N$. 
\end{defi}

We are now ready to state an important property of the structures that are  infinitely generic in the multiverse. 

\begin{prop}\label{excomp}
Every structure that is  infinitely generic in the multiverse is existentially complete in $\mathbb{M}_{M}^\Gamma$.  
\end{prop}
\begin{proof}
Let $N$ be a structure that is infinitely generic in the multiverse and let $\varphi$ be a $\Sigma_1$-sentence with parameters in $N$ and
true in some $N'$ such that $N' \leq_{\mathbb{M}_{M}^\Gamma} N$. Then, by Theorem \ref{generic}, there is a model $N''$ that 
is infinitely generic in the multiverse and such that  $N'' \leq_{\mathbb{M}_{M}^\Gamma} N'$. Since $\varphi$ is $\Sigma_1$ and holds
in $N'$, then $\varphi$ holds in $N''$, too. Now, applying Proposition \ref{geneq}, we obtain that $\varphi$ holds in $N$. 
\end{proof}

As the rest of the section will show, infinitely generic structures in $\mathbb{M}_{M}^\Gamma$ display a very strong level of generic absoluteness with respect to
generic extensions by posets in $\Gamma$.  In \cite{Bag00},  Bounded Forcing Axioms  have been characterized in terms of generic absoluteness,
 showing that $BFA(\kappa, \Gamma)$ is equivalent to say that, for all $\mathbb{P} \in \Gamma$,  $H_{\kappa^+} \prec_{\Sigma_1} H_{\kappa^+}^{V^\mathbb{P}}$. 

\begin{coro}
Let $N$ be an infinitely generic structures in $\mathbb{M}_{M}^\Gamma$, then for all $\kappa$,  $BFA(\kappa, \Gamma)$ holds in $N$. \qed
\end{coro}

Being existentially complete in $\mathbb{M}_{M}^\Gamma$ is enough
for having absoluteness with respect to all $\Pi_2$-sentences. 

\begin{prop}\label{Pi2}
If $N$ is existentially complete in $N'$ and $\varphi$ is a $\Pi_2$-sentence with parameters in $N$ which holds in $N'$, 
then $\varphi$ holds in $N$. 
\end{prop}
\begin{proof}
Let $\varphi = \forall x_1, \ldots, x_n \psi(x_1, \ldots, x_n)$, with $\psi$ being $\Sigma_1$. Let $a_1, \ldots a_n$ be arbitrary 
but fixed elements of $N$. Then, since $N' \leq_{\mathbb{M}_{M}^\Gamma} N$, we have that $N \subseteq N'$ and therefore that 
$\varphi$ can be seen as having parameters in $N'$ also. Since $\varphi$ holds in $N'$, we have that $\psi(a_1, \ldots, a_n)$
also holds in $N'$. Being $N$ existentially complete in $N'$ and $\psi$ being $\Sigma_1$, we get that $N \models \psi(a_1, \ldots, a_n)$.
By arbitrariness of the choice of $a_1, \ldots a_n$ we have that $N \models \forall x_1, \ldots, x_n \psi(x_1, \ldots, x_n)$ must be
the case. Hence $\varphi$ holds in $N$.
\end{proof}

\begin{coro}
If $N$ is existentially complete in $\mathbb{M}_{M}^\Gamma$ and $\varphi$ is a $\Pi_2$-sentence with parameters in $N$, that holds in $N'$ 
for some $N'$ such that $N' \leq_{\mathbb{M}_{M}^\Gamma} N$, then $\varphi$ holds in $N$. \qed
\end{coro}

\begin{defi}
Given $N$ and $N'$ such that $N' \leq_{\mathbb{M}_{M}^\Gamma} N$, $N$ is said to be  \emph{$\Pi_2$-complete in $N'$}, 
if every $\Pi_2$-sentence with parameters in $N$ and true in $N'$ is also true in $N$. Moreover, we say that $N$ 
is \emph{$\Pi_2$-complete in $\mathbb{M}_{M}^\Gamma$} if $N$ is $\Pi_2$-complete in $N'$ for every $N'$ such that
$N' \leq_{\mathbb{M}_{M}^\Gamma} N$. 
\end{defi}

\begin{prop}\label{excomp}
Every structure that is  infinitely generic in the multiverse is $\Pi_2$-complete in $\mathbb{M}_{M}^\Gamma$.  \qed
\end{prop}

Besides these closure properties, infinitely generic structures capture an even more fundamental aspect of generic absoluteness:
the possibility that a formula is persistent across the generic multiverse. 

\begin{teo}\label{persistent}
Let $\varphi$ be a formula with parameters in $N$, with $N$ infinitely generic in the multiverse. If $\varphi$ holds in some $N' \supseteq N$ and 
$\forall N'' \supseteq N'(N'' \models \varphi)$, then $N \models \varphi$. 
\end{teo}
\begin{proof}
We know that $N' \models \varphi$. Now extend $N'$ to an $N''$ that is  infinitely generic in the multiverse. By hypothesis $N'' \models \varphi$, 
but now, by Proposition \ref{geneq}, we get that $N \models \varphi$. 
\end{proof}

Notice that, in modal terms, Theorem \ref{persistent} implies that an infinitely generic structure in the multiverse 
verifies the modal principle $\Diamond\Box \varphi \to \varphi$, where $\Box\varphi$ is defined as ``$\varphi$ holds in all generic 
extensions by forcing in $\Gamma$''. 

Interestingly enough this principle has been studied both in the context of the modal logic
of forcing and in the context of generic absoluteness principles. In the former case has been named Maximality Principle  \cite{HamMP}, 
while in the latter case represents a weak form of Resurrection Axiom \cite{HamJohn}. Indeed, 
$\Diamond\Box \varphi \to \varphi$ is modally equivalent to  $\varphi \to \Box \Diamond \varphi$, which in turns expresses the fact that 
if a formula $\varphi$ is valid in a model $M$, then, whatever is the generic extension of $M$ we consider, it is possible to further extend it
to a new model where $\varphi$ holds. 

Hamkins' Resurrection Axiom is actually stronger, being in the form of an equivalence. 

 \begin{defi}(Hamkins, \cite{HamJohn})
 The Resurrection Axiom RA($\Gamma$) asserts that for every forcing $\mathbb{Q} \in \Gamma$, there is a $\mathbb{Q}$-name $\dot{\mathbb{R}}$ with
 $1_\mathbb{Q} \Vdash \dot{\mathbb{R}} \in {\Gamma}$ and  such that $H_{2^{\aleph_0}} \prec H_{2^{\aleph_0}}^{V^{\mathbb{Q} \ast \dot{\mathbb{R}}}}$. 
 \end{defi}

Using Theorem \ref{persistent} we get the following corollary.

\begin{coro}
Let $N$ be an infinitely generic structure in $\mathbb{M}_{M}^\Gamma$, then if $H_{2^{\aleph_0}}^N \models \varphi$, 
then, for every forcing $\mathbb{Q} \in \Gamma$, there is a $\mathbb{Q}$-name $\dot{\mathbb{R}}$ with
 $1_\mathbb{Q} \Vdash \dot{\mathbb{R}} \in {\Gamma}$ such that $H_{2^{\aleph_0}}^{N^{\mathbb{Q} \ast \dot{\mathbb{R}}}} \models \varphi$. \qed
\end{coro}

Notice that Theorem \ref{persistent} allows for a stronger version of the above corollary; one that avoids the restriction of the parameters to $H_{2^{\aleph_0}}$.  In other words, with the terminology of \cite{HamMP}, we have that an infinitely generic structure in $\mathbb{M}_{M}^\Gamma$ models the boldface version of the Maximality Principle restricted to forcings in $\Gamma$.

The question, then, is whether there are multiverses such that the properties of $\Gamma$ allow to show that $\mathbb{M}_{M}^\Gamma$ is $\sigma$-closed. Next section 
will provide such examples.

\subsection{Examples}

We now turn our attention to the forcing $\mathbb{M}_{M}^C$. Contrary to the fact
that there are sequences of length $\omega$ of Cohen forcing extensions that are non-amalgamable,\footnote{See Observation 33 of \cite{Hamkinsgeology}.} the following theorem shows that nonetheless it is possible to find lower bounds for countable sequences of elements of $\mathbb{M}_{M}^C$. 

The following theorem was proved in a series of conversations with Joel David Hamkins, in 2011, at the Young Set Theory Workshop in Bonn and continued at the summer school on set theory and second order logic, held in London during the same year.\footnote{For completeness we report the proof that already appeared in \cite{HamUP}.}

\begin{teo}(Hamkins, V.)\label{HandV}
Let $M$ be a countable transitive model of ZFC and let $\langle c_n : n \in \omega\rangle$ 
be a tower of (finitely) mutually $M$-generic Cohen reals
$$
M \subseteq M[c_0] \subseteq M[c_0][c_1] \subseteq \ldots
$$
Then there is a $M$-generic Cohen real $d$ such that $c_n \in M[d]$. As a consequence  
$M[c_0][c_1]\ldots[c_n] \subseteq M[d]$. Furthermore, $d$ is $M[c_0, \ldots c_n]$-generic. 
\end{teo}
\begin{proof}

The idea of the construction consists of making finitely many changes to each real $c_n$, resulting in a real $d_n$, in such a way that the resulting combined real $d$ is $M$-generic for the forcing $Add(\omega, \omega)$. To do this we construct $d$ as a binary function $d: \omega \times \omega \rightarrow 2$, where the $n$th slice will correspond to a real $d_n: \omega \rightarrow 2$. 

To this aim, fix an enumeration $\langle D_n : n \in \omega \rangle$ of all open dense sets of $Add(\omega, \omega)$ in $M$. 
We construct $d$ in stages, by defining a sequences $\langle d_n : n \in \omega \rangle$ of reals and another auxiliary
sequence $\langle p_n : n \in \omega \rangle$ of finite conditions in $Add(\omega, \omega)$ that will guarantee the genericity of $d$. 
Suppose we are at  stage $n$ and we  have completely specified $d_k$, for $k < n$, and we are also committed to a finite condition $p_{n-1} \in Add(\omega, \omega)$ such that $p_{n-1} \restriction n \subseteq d_0 \times \ldots \times d_{n-1}$.
We now want to get into the dense set $D_n$. By inductive hypothesis $d_0 \times \ldots \times d_{n-1}$ is an $M$-generic real, since each $d_k$ is a finite modification of the corresponding $c_k$ which are mutually $M$-generic. Now, choose a Cohen real $e$ that is $M[d_0 \times \ldots \times d_{n-1}]$-generic
for $Add(\omega, \omega)$ and such that $p_{n-1} \subseteq d_0 \times \ldots \times d_{n-1} \times e$; this is possible since compatibility with $p_{n-1}$
amounts in finite modifications that preserve genericity. Notice that $d_0 \times \ldots \times d_{n-1} \times e$ is $M$-generic for $Add(\omega, \omega)$. Consequently there is a finite $p \subseteq d_0 \times \ldots \times d_{n-1} \times e$ in $D_n$. Since both $p$ and $p_{n-1}$ are included 
in $d_0 \times \ldots \times d_{n-1} \times e$ we have that $p \cup p_{n-1}$ is a condition that, since $D_n$ is open, is in $D_n$. Call this union $p_n$. 
We then finitely modify $c_n$ to a new condition $d_n$, so to be compatible with $p_n$; hence $p_n \restriction (n +1) \subseteq d_0 \times \ldots \times d_{n}$. Therefore, defining $d = \bigtimes_{n \in \omega} d_n$, we will have that $d$ will eventually contain
all conditions $p_n$ and so will be $M$-generic for $Add(\omega, \omega)$. Indeed every $p_n$ is finite and at each step of our construction we build $d$ in order to include more and more initial coordinates of a given $p_n$. Moreover, $c_n \in M[d]$, since each $c_n$ is a finite modification of $d_n$. Hence $M[c_0][c_1] \ldots [c_n] \subseteq M[d]$. 
\end{proof}

Since the product of countably many copies of $Add(\omega, 1)$ is isomorphic to $Add(\omega, 1)$, we get the following corollary. 

\begin{coro}
The forcing $\mathbb{M}_{M}^C$ is $\sigma$-closed. Therefore, infinitely generic structures in the multiverse $\mathbb{M}_{M}^C$ exist. \qed
\end{coro}

However, when considering larger classes of posets, like $\sigma$-closed, proper or semiproper posets, it is easy to see that the countability of the model $M$
 represents a serious obstacle for the construction of infinite generic models in $\mathbb{M}_{M}^{\Gamma}$. This is because in countably
many steps we can reach the hight of the model and thus we risk to collapse everything to countable or to violate the powerset axiom. 
Next theorem shows that to avoid these phenomena it is sufficient to put a bound on the posets allowed.

\begin{teo}(\cite{Hamkinsgeology}, Theorem 34)\label{countablechain}
Suppose that $M$ is a countable model of ZFC, and 
$$
M[G_0] \subseteq M[G_1] \subseteq  \ldots \subseteq M[G_n] \subseteq \ldots
$$
is an increasing sequence of forcing extensions of $W$, with $G_n \subseteq \mathbb{Q}_n \in M$
being $M$-generic. If the cardinalites of the $\mathbb{Q}_n$'s in $M$ are bounded in $M$, then 
there is a set-forcing extension $M[G]$ with $M[G_n] \subseteq M[G]$, for all $n < \omega$. \qed 
\end{teo}

Then
we have the following schema of theorems.

\begin{teo}\label{restricted}
The forcing $\mathbb{M}_{M}^{\kappa}$ is $\sigma$-closed. Therefore, infinitely generic structures in the multiverse $\mathbb{M}_{M}^{\kappa}$ exist. 
\end{teo}
\begin{proof}
It is sufficient to apply Theorem \ref{countablechain} to any $\omega$-chain in $\mathbb{M}_{M}^{\kappa}$, noting that the proof of Theorem 34 in \cite{Hamkinsgeology} 
shows that when 
the $G_n$'s are mutually generic, it is possible to build  a generic extension witnessing the closure of the chain by a filter $G$  that is also generic 
with respect to each $G_n$. 
\end{proof}  
  
 Notice that Theorem  \ref{restricted} is optimal, because of the following result.

 \begin{teo}(\cite{HamUP} Theorem 12)
Suppose that $M$ is a countable model of ZFC, and 
$$
M[G_0] \subseteq M[G_1] \subseteq  \ldots \subseteq M[G_n] \subseteq \ldots
$$
is an increasing sequence of forcing extensions of $W$, with $G_n \subseteq \mathbb{Q}_n \in M$
being $M$-generic. The following are equivalent:
\begin{enumerate}
\item The chain is bounded above by a forcing extension $M[H]$, by some forcing notion $\mathbb{Q} \in M$
and an $M$-generic filter $H \subseteq \mathbb{Q}$.
\item Letting $\kappa_n$ be the smallest size, in $M$, of the Boolean completion of a forcing notions $\mathbb{P}_n$, for which there is an
$M$-generic filter $H_n \subseteq \mathbb{P}_n$ such that $M[G_n] = M[H_n]$, we have that $sup\{\kappa_n :  n \in \omega\} \in M$.
\end{enumerate}

 If the cardinalites of the $\mathbb{Q}_n$'s in $M$ are bounded in $M$, then 
there is a set-forcing extension $M[G]$ with $M[G_n] \subseteq M[G]$, for all $n < \omega$. \qed 
\end{teo}

\subsection{Open problems}
 
Since MA$(\kappa)$ is equivalent to MA$(\kappa)$ restricted to postes of cardinality $\leq \kappa$, Theorem \ref{restricted} suggests the following interesting question:

\medskip

\begin{quote}
\emph{Question 1}: 
Letting $\Gamma_\kappa(c.c.c.)$ be the class of c.c.c. posets of cardinality $\leq \kappa$, 
is $\mathbb{M}_{M}^{\Gamma_\kappa(c.c.c.)}$  $\sigma$-closed?
\end{quote}

\medskip

More in general, defining BFA$(\Gamma_\kappa)$ to be the weakening of BFA$(\kappa, \Gamma)$ where a poset in $\Gamma$
is assumed to  have cardinality $\leq \kappa$, the following question naturally arises:

\medskip

\begin{quote}
\emph{Question 2}:
Can Robinson's infinite forcing give a consistency proof of the corresponding  BFA$(\Gamma_\kappa)$? 
\end{quote}
 
 \section{Concluding remarks}

In this paper we showed an interesting connection between generic absoluteness and Robinson infinite forcing.
One could think that the reason
is the close connection between Cohen's and Robinson's techniques. However, the possibility to recast
Robinson's infinite forcing by means of model theoretic constructions that do not involve forcing suggests 
otherwise. Indeed,  Cherlin showed \cite{Cherlin} that infinitely generic structures can be 
approximated by chains of models\footnote{See \cite{HW} for other approximation methods by E. Fisher, 
H. Simmons, or J. Hirschfield. We decided to mention only Cherlin, since the notion of 
persistent formulas bears an obvious connection with generic absoluteness.} that inductively capture larger and larger classes of persistent sentences, 
without introducing any notion coming form the theory of forcing. 

If Cohen's and Robinson's techniques are fairly different from a theoretical perspective, however the results of this paper
suggest that the strategy for overcoming independence in set theory---i.e. generic absoluteness results---can
be seen as a particular case of a model theoretic technique applied to the multiverse. In a slogan, Robinson's ideas offer
a revenge against Cohen's independence results. 

Therefore, the extent and the nature of the link between generic absoluteness and infinite forcing 
remains a promising question that still requires a complete answer. We started here by 
scratching the surface of an interesting interplay between set theory and model theory. 

The obvious extension of the results of this paper concerns the possibility to overcome the limits imposed by the height of the models of a generic 
multiverse---which, by Cohen's method, is always the same of the ground model. 
Indeed, we might modify the definition of  $\mathbb{M}_{M}^{\Gamma}$ so that
at each step the hight of the models increases. An obvious attempt could use generic ultrapowers; 
in that case it would be interesting to understand the connection between this kind of multiverse and Woodin's $\mathbb{P}_{max}$. 
We leave this task for future work.

 \bigskip
 
 {\small{\textbf{Acknowledgements.} We thank an anonymous referee for the careful reading, comments, and criticisms. We acknowledge the kind support of FAPESP in the form of the Jovem Pesquisador grant n. 2016/25891-3.}}


\begin{thebibliography}{99}


\bibitem{FriedmanArrigoni} T. Arrigoni and S. D. Friedman, The hyper universe program, \emph{Bulletin of Symbolic Logic}, 19(1), pp. 77--96, 2013.
\bibitem{Bag00} J. Bagaria, Bounded forcing axioms as principles of generic absoluteness, \emph{Archive for Mathematical Logic}, 39(6), pp. 393--401, 2000. 
\bibitem{Balaguer} M. Balaguer, \emph{Platonism and Anti-platonism in Mathematics}, Oxford University Press, 1998.
\bibitem{Cherlin} G. Cherlin, \emph{A New Approach to the Theory of Infinitely Generic Structures}, Dissertation, Yale University, 1971. 
\bibitem{Cohen2005} P. Cohen, Skolem and pessimism about proof in mathematics, \emph{Philosophical Transaction of the Royal Society A}, 363, pp. 2407--2418. 2005.
\bibitem{Hamkinsgeology} G. Fuchs and J. D. Hamkins and J. Reitz, Set-theoretic geology, \emph{Annals of Pure and Applied Logic}, 166(4), pp. 464--501, 2015.
\bibitem{HamMP} J. D. Hamkins, A simple maximality principle, \emph{Journal of Symbolic Logic}, 68(2), pp. 527--550, 2003. 
\bibitem{Hamkins} J. D. Hamkins, The set-theoretic multiverse, \emph{Review of Symbolic Logic}, 5, pp. 416--449, 2012.
\bibitem{HamUP} J. D. Hamkins, Upward closure and amalgamation in the generic multiverse of a countable model of set theory, \emph{RIMS Kyôkyûroku}, pp. 17--31, 2016. 
\bibitem{HamJohn} J. D. Hamkins and T. Johnstone, Resurrection axioms and uplifting cardinals, \emph{Archive for Mathematical Logic}, 50(3-4), pp. 463--485, 2014. 
\bibitem{HW} J. Hirschfeld and W. H. Wheeler, \emph{Forcing, Arithmetic, Division Rings}, Springer, 1975. 
\bibitem{Hodges} W. Hodges, \emph{Building Models by Games}, Cambridge University Press, 1985. 
\bibitem{Kunen} K. Kunen \emph{Set Theory. An Introduction to Independence Proofs}, North-Holland, 1980. 
\bibitem{Magidor} M. Magidor, Some set theories are more equal then others, First draft from the EFI Project website, 2012. 
\bibitem{Manevitz} L. Manevitz, Robinson forcing is not absolute \emph{Israel Journal of Mathematics}, 25, pp. 211--232, 1976. 
\bibitem{Vaananenmultiverse} J. V\"a\"an\"anen, Multiverse set theory and absolutely undecidable propositions, in \emph{Interpreting G\"odel} (J. Kennedy ed.), pp. 180-208, Cambridge University Press, 2014. 
\bibitem{Viale2} M. Viale, Category forcings, MM$^{+++}$, and generic absoluteness for the theory of strong forcing axioms, \emph{Journal of the American Mathematical Society}, 29(3), pp. 675--728, 2016. 
\bibitem{wooBMMc} H. Woodin, \emph{The axiom of Determinacy, Forcing Axioms, and the Nonstationary Ideal}, Walter de Gruyter and Co., 1999.
\bibitem{WoodinInfinite} H. Woodin, The Realm of the Infinite, in \emph{Infinity. New Research Frontiers} (M. Heller and H. Woodin eds.), pp. 89--118, Cambridge University Press, 2011. 



\end{thebibliography}
\end{document}